\documentclass{elsarticle}

\usepackage[top=3cm, bottom=3cm,left=3cm,right=3cm]{geometry}

\usepackage{amsmath, amsthm}
\usepackage{amscd}
\usepackage{amssymb}

\usepackage{palatino}
\usepackage{mathpazo}

\usepackage{enumerate}

\usepackage{mathrsfs}
\usepackage{hyperref}
\newtheorem{theorem}{Theorem}[subsection]

\newtheorem{stmt}[theorem]{}

\newtheorem{remark}[theorem]{Remark}

\newtheorem{definition}[theorem]{Definition}

\newcommand{\sconn}{sc}
\newcommand{\res}{res}

\newcommand{\Ext}{Ext}

\newcommand{\GL}{GL}

\newcommand{\SOrth}{SO}

\newcommand{\Symp}{Sp}
\newcommand{\SL}{SL}
\newcommand{\Sym}{Sym}
\newcommand{\Lie}{Lie}
\newcommand{\Stab}{Stab}
\newcommand{\End}{End}
\newcommand{\Ad}{Ad}
\newcommand{\Int}{Int}

\newcommand{\Gsc}{G_{\sconn}}
\newcommand{\exterior}{{\bigwedge}}

\newcommand{\tensor}{\otimes}

\newcommand{\Z}{\mathbf{Z}}
\newcommand{\Q}{\mathbf{Q}}
\newcommand{\G}{\mathbf{G}}
\newcommand{\Gm}{\G_m}

\newcommand{\LL}{\mathscr{L}}

\newcommand{\supp}{Supp}

\newcommand{\goodwt}{\mathcal{W}}

\newcommand\wts[1]{{ X_#1 }}
\newcommand\dom[1]{{ X_{#1}^+ }}
\newcommand{\ho}[2]{\nabla_{#1}(#2)}

\newcommand{\chis}[1]{{ \chi_S( #1 ) }}
\newcommand{\chih}[1]{{ \chi_H( #1 ) }}

\begin{document}

 \title{Title}
 \author{Chuck Hague\corref{cor1}}
%  \author{Chuck Hague\corref{cor1}\fnref{label1}}
 \ead{hague@math.udel.edu}
 \ead[url]{http://sites.google.com/site/chuckhague/}
 \cortext[cor1]{Corresponding author}
 \address{Department of Mathematical Sciences \\
University of Delaware \\
501 Ewing Hall \\
Newark, DE 19716\fnref{label1}}
 
  \author{George McNinch\fnref{label2}}
 \ead{george.mcninch@tufts.edu}
 \ead[url]{http://gmcninch.math.tufts.edu/}
 \fntext[label2]{Research of McNinch supported in part by the US NSA award H9230-11-1-0164.}
 \address{Department of Mathematics \\ 
Tufts University \\ 
503 Boston Ave \\ 
Medford, MA 02155\fnref{label2}}

\title{Some good-filtration subgroups of simple algebraic groups}

\begin{abstract}
Let $G$ be a connected and reductive algebraic group over an
algebraically closed field of characteristic $p>0$.  An interesting
class of representations of $G$ consists of those $G$-modules having
a \emph{good filtration} -- i.e. a filtration whose layers are the
induced highest weight modules obtained as the space of
global sections of $G$-linearized line bundles on the flag variety
of $G$. Let $H \subset G$ be a connected and reductive subgroup of
$G$. One says that $(G,H)$ is a \emph{Donkin pair}, or that $H$ is a
\emph{good filtration subgroup} of $G$, if whenever the $G$-module
$V$ has a good filtration, the $H$-module $\res^G_H V$ has a good
filtration.

In this paper, we show when $G$ is a ``classical group'' that the
\emph{optimal} $\SL_2$-subgroups of $G$ are good filtration
subgroups. We also consider the cases of subsystem subgroups in all
types and determine some primes for which they are good filtration
subgroups.
\end{abstract}


\maketitle

\setcounter{tocdepth}{1}
\tableofcontents

\section{Introduction}
\label{sec:introduction}

Let $G$ be a connected and reductive group over an algebraically
closed field of characteristic $p>0$, and let $H \subset G$ be a
closed subgroup which is also connected and reductive.  We are
concerned here with the linear representations of the algebraic group
$G$ -- i.e. with $G$-modules -- and with their restriction to $H$.

Of particular interest are the induced $G$-modules $\ho G \lambda$
and the induced $H$-modules $\ho H \lambda$ obtained as global sections of equivariant bundles on the associated flag varieties; see \S\ref{induced modules}.
One says that $H$ is a \emph{good-filtration subgroup} of $G$ -- or
that $(G,H)$ is a \emph{Donkin pair} -- provided that for any induced
$G$-module $V$ the $H$-module $\res^G_HV$ obtained from $V$ by
restriction to $H$ has an exhaustive filtration whose successive
quotients are induced $H$-modules.

Donkin proved in \cite{donkin} -- under some mild assumptions on the
characteristic -- that a Levi factor of a parabolic subgroup of $G$ is
always a good-filtration subgroup; subsequently, Mathieu gave an
unconditional proof \cite{mathieu} of this result using the geometric
method of Frobenius splitting (cf. also the accounts in \cite{BK} \S4, \cite{JRAG} Ch. G, and \cite{vdK2}).

In \cite{brundan}, Brundan proved that a large class of reductive
spherical subgroups of $G$ are good filtration subgroups, under mild
restrictions on $p$; recall that a subgroup $H$ is said to
\emph{spherical} if there is a dense $H$-orbit on the flag variety
$G/B$ of $G$.  In that paper, Brudan also conjectured that $H$ is a
good filtration subgroup if either (i) $H$ is the centralizer of a
graph automorphism of $G$, or (ii) $H$ is the centralizer of an
involution of $G$ and $p > 2$.  Brundan's conjecture is now a theorem;
many cases were covered already in \cite{brundan} and the remaining
cases were handled by van der Kallen in \cite{vdK}.

In this paper we extend the study of Donkin pairs to more reductive
subgroups of $G$. In particular, we consider two classes of reductive
subgroups: optimal $SL_2$-subgroups and the so-called subsystem
subgroups. In \S\ref{section:preliminaries} we give preliminaries on algebraic groups, good
filtrations, and optimal $SL_2$ subgroups.

In this paper a group of \emph{classical type}, or just a \emph{classical group}, will be a group isomorphic to $SL(V)$ or the stabilizer of a nondegenerate alternating or bilinear form $\beta$ when $p > 2$. In \S \ref{sec:the groups of interest} we give a general criterion for a reductive subgroup of a group of classical type to be a good filtration subgroup (\ref{check-Donkin-pair-on-exterior-powers}). Since a group of classical type is not simply-connected when the form $\beta$ is symmetric we also consider the simply-connected covers of these groups (\ref{donkin-subgroups-of-spin}). We then give a criterion for checking when a reductive subgroup of a group of exceptional type is a good filtration subgroup (Theorem \ref{th:exceptional group Donkin pairs}).

In \S\ref{sec:main results} we give our main results. In \S\ref{sub:optimal sl2 subgroups of classical groups} we consider the case in which $G$ is a classical group and $S \subset G$ is an \emph{optimal $SL_2$ subgroup}, a notion essentially due to Seitz \cite{seitz}; we
follow the characterization of these subgroups given in \cite{mcninch-optimal}. The main result of this section is Theorem \ref{classical-optimal-donkin-pair}, which states that
optimal $SL_2$-subgroups of classical groups are good filtration
subgroups. Our proof is modeled on arguments of Donkin from \cite{donkin}. We also consider optimal $SL_2$ subgroups of the simply-connected covers of classical groups (Theorem \ref{spin-group-optimal-donkin-pair}). In \S\ref{sub:optimal sl2s in the exceptional case} we consider optimal $SL_2$ subgroups of exceptional groups. In these theorems we crucially use induction arguments which reduce to the case of a \emph{distinguished} optimal $SL_2$ subgroup.

Recall that a \emph{subsystem subgroup} of $G$ is a connected semisimple
subgroup which is normalized by a maximal torus. In \S\ref{sub:subsystem subgroups in the exceptional case} we consider arbitrary subsystem subgroups of semisimple groups. Since Brundan's conjecture implies that every subsystem subgroup of a group of type $A$, $B$, $C$, or $D$ is a good filtration subgroup when $p > 2$, we only consider the exceptional case. By the transitivity of the good filtration subgroup property and the fact that Levi factors of parabolic subgroups are good filtration subgroups, it suffices to consider only the case where the subsystem subgroup is of maximal rank (= rank $G$). The main result in this section is Theorem \ref{th:subsystem subgroups in the exceptional case}, which gives primes $p$ for which the maximal rank reductive subgroups not already covered by the Brundan conjecture are good filtration subgroups.

Also, we would like to thank the anonymous referee for suggesting useful improvements.

%. It also suffices to consider only the case of maximal rank (= rank $G$) subsystem subgroups. In Theorem \ref{th:exceptional group Donkin pairs} we give sufficient (although
%not, a priori, necessary) conditions for a maximal rank reductive subgroup of a
%group of exceptional type to be a good filtration subgroup. This
%theorem can be considered as a variant of Theorem
%\ref{check-Donkin-pair-on-exterior-powers} valid for exceptional
%groups.

%
%In \S\ref{subsub:optimal sl2s in the exceptional case} we apply
%Theorem \ref{th:exceptional group Donkin pairs} to the case of certain
%optimal $SL_2$ subgroups of groups of exceptional type.

%% Of course, when $G$ is simply connected $C_G(s)$ is connected for
%% each semisimple $s \in G$. 
%In \S \ref{sub:subsystem subgroups in the exceptional case} we apply Theorem
%\ref{check-Donkin-pair-on-exterior-powers} to certain subsystem
%subgroups of groups of exceptional type.

\section{Preliminaries}
\label{section:preliminaries}

\subsection{Induced modules for reductive groups}
\label{induced modules}
Let $k$ be an algebraically closed field of positive characteristic $p$ and let $G$ be a connected and reductive algebraic group over $k$. Fix a maximal torus $T \subset
G$, and choose a Borel subgroup $B \subset G$ containing $T$. For us, a representation of a linear algebraic group always means a rational representation; namely, a co-module for the coordinate algebra.

We write $X^*(T)$ for the character group and $X_*(T)$ for the
co-character group of the torus $T$.  We write $(\lambda,\phi) \mapsto
\langle \lambda,\phi \rangle \in \Z$ for the natural pairing $X^*(T)
\times X_*(T) \to \Z$. Recall that the choice of the Borel subgroup
$B$ determines a system of positive roots $R^+$ of the set of roots $R
\subset X^*(T)$.

Each character $\lambda \in X^*(T)$ determines a $G$-linearized line
bundle $\LL(\lambda)$ on the flag variety $G/B$. The group $G$ acts
linearly on the space of global sections
\begin{equation*}
H^0(G/B,\LL(\lambda));
\end{equation*}
we write $\ho{G}{\lambda}$ for this $G$-module (which is denoted
$H^0_G(\lambda) = H^0(\lambda)$ in \cite{JRAG} \S II). Then
$\ho{G}{\lambda}$ is non-0 if and only if $\lambda$ is
\emph{dominant}; i.e. if and only if $\langle \lambda,\alpha^\vee
\rangle \ge 0$ for each $\alpha \in R^+$. The representations
$\ho{G}{\lambda}$ are known as \emph{induced modules} for $G$.

Assume that $G$ is quasisimple; in this case, we number the nodes of
the Dynkin diagram of $G$ -- and hence the simple roots and
fundamental dominant weights -- as in Bourbaki \cite{bourbaki}, Plate
 I-IX. Let $\varpi_i \in X^*(G) \tensor \Q$ denote the
\emph{fundamental dominant weights}; if $\alpha_1,\dots,\alpha_r$ are
the simple roots with corresponding co-roots $\alpha_i^\vee \in
X_*(T)$, then $\langle \varpi_i,\alpha_j^\vee \rangle = \delta_{ij}$
for $1 \le i,j \le r$.  Of course, $G$ is \emph{simply connected} if
and only if $\varpi_i \in X^*(T)$ for $1 \le i \le r$.

\subsection{Modules with a good filtration} \label{subsub:modules with a good filtration}

Let $V$ be any $G$-module.  A collection of $G$-submodules $V_i
\subset V$ for $i \in \Z_{\ge 0}$ forms a \emph{filtration} of $V$
provided that $V_i \subset V_{i+1}$ for $i \ge 0$ and that $V =
\bigcup_{i \ge 0} V_i$.  The \emph{layers} of the filtration are the
quotient modules $V_i/V_{i-1}$.

The filtration of $V$ is said to be a \emph{good filtration} if for
each $i \ge 1$, the layer $V_i/V_{i-1}$ is either $0$ or is isomorphic to an induced
module $\ho{G}{\lambda_i}$ for some dominant weight $\lambda_i$.

For a $G$-module $V$ with a good filtration, the \emph{support} of $V$ (written as $\supp(V)$) is
the set of $\lambda \in \dom G$ for which $\ho{G}{\lambda}$ occurs as a
layer in a good filtration of $V$. It follows from
\cite{JRAG}, Prop. II.4.16 that the support of $V$ is independent of
the choice of good filtration of $V$.

\begin{stmt}
\label{direct-summand-good-filt}
Let $(*) \quad 0 \to V \to E \to W \to 0$ be a short exact sequence of
$G$-modules.
\begin{enumerate}[(a)]
\item Assume that $V$ has a good filtration. Then $E$ has a good
 filtration if and only if $W$ has a good filtration.
\item If the sequence $(*)$ is \emph{split exact}, and if $E$ has a
 good filtration, then both $V$ and $W$ have a good filtration.
\end{enumerate}
\end{stmt}

\begin{proof}
 Assertion (a) follows from the ``homological'' characterization of
 good filtrations found in \cite{JRAG}, Prop. II.4.16,
 and (b) is an immediate consequence of (a).
\end{proof}

We also observe the following:
\begin{stmt}
\label{good-filtration-filtration}
If the $G$-module $V$ has a filtration for which each quotient
$V_i/V_{i-1}$ has a good filtration for $i \ge 1$, then $V$ has a
good filtration.
\end{stmt}

\begin{proof}
This is straightforward when $V$ is finite dimensional; the general
case is obtained in \cite{donkin}, Prop. 3.1.1.
\end{proof}

The following important result was first obtained for $p \gg 0$ by J. Wang, with improvements to the prime $p$ by Donkin \cite{donkin} (under some small restrictions), and in general by
Mathieu \cite{mathieu}.

\begin{stmt}[Wang, Donkin, Mathieu]
 \label{tensor-good-filtration} If $V$ and $W$ are finite dimensional
 $G$-modules each having a good filtration, then the $G$-module $V
 \tensor W$ has a good filtration.
\end{stmt}

We also have the following useful fact.
\begin{stmt}
\label{tensor-induced-quotient}
Let $\lambda,\mu \in \dom G$. There is a surjective mapping of
$G$-modules $\ho{G}{\lambda} \tensor \ho{G}{\mu} \to \ho{G}{\lambda +
 \mu}$ whose kernel has a good filtration.
\end{stmt}

\begin{proof}
It follows from \ref{tensor-good-filtration} that $M = \ho{G}{\lambda} \tensor \ho{G}{\mu}$ has a good filtration. Since $\lambda + \mu$ is the highest weight of $M$ we have $\lambda + \mu \in \supp(M)$. Since any weight $\gamma$ of $M$ satisfies $\gamma \le \lambda + \mu$, it follows that for $\sigma \in \supp(M) \setminus \{ \lambda + \mu \}$ we have $\sigma < \lambda+\mu$ and hence $\Ext^1_G(\ho{G}{\sigma},\ho{G}{\lambda + \mu}) = 0$ \cite{JRAG}, Prop. II.6.20.

Since $\ho G {\lambda + \mu}$ occurs as a layer in a good filtration of $M$, there is a submodule $V \subseteq M$ with a surjection $ f : V \twoheadrightarrow \ho G {\lambda + \mu} $. By the above, $\Ext^1_G(M/V,\ho{G}{\lambda + \mu}) = 0$, so $f$ lifts to a surjection $\widehat f : M \twoheadrightarrow \ho G {\lambda + \mu}$, and the kernel of $\widehat f$ has a filtration with layers $ \ho G \sigma $ for $\sigma \in \supp(M) \setminus \{ \lambda + \mu \}$.
\end{proof}

\begin{stmt}
\label{multilinear}
Let $V$ be a $G$-module with a good filtration. Then $V^{\otimes m}$
has a good filtration for all $m \geq 0$. If $m < p$ then
\begin{equation*}
 \exterior^m V \quad \text{and} \quad \Sym^m  V
\end{equation*}
each have a good filtration.
\end{stmt}

\begin{proof}
 The observation that $V^{\tensor m}$ has a good filtration follows from
 \ref{tensor-good-filtration}. If $m < p$, it is sufficient by
 \ref{direct-summand-good-filt} to argue that $\exterior^m V$ and
 $\Sym^m V$ are direct summands of $V^{\otimes m}$ as $G$-modules. Although this fact is well-known we give the proof here for completeness. For this, it suffices to observe that there are $G$-linear
 splittings $\sigma_1:\exterior^m V \to V^{\tensor m}$ and
 $\sigma_2:\Sym^m V \to V^{\tensor m}$ of the natural surjections
\begin{equation*}
  \pi_1:V^{\tensor m} \to \exterior^m V  \quad \text{and} \quad
  \pi_2:V^{\tensor m} \to \Sym^m V;
\end{equation*}
i.e. $\pi_i \circ \sigma_i = \operatorname{id}$ for $i=1,2$. We
define the required $\sigma_i$ as follows; for $v_1,\dots,v_m \in V$
one sets
\begin{equation*}
  \sigma_1(v_1 \wedge \cdots \wedge v_m) = \dfrac{1}{m!}\sum_{\tau \in \Sym_m} 
  \operatorname{sgn}(\tau) v_{\tau(1)} \tensor \cdots \tensor v_{\tau(m)}
\end{equation*}
and
\begin{equation*}
  \sigma_2(v_1  \cdots  v_m) = \dfrac{1}{m!}\sum_{\tau \in \Sym_m} 
   v_{\tau(1)} \tensor \cdots \tensor v_{\tau(m)}
\end{equation*}
where $\Sym_m$ is the symmetric group on $m$ letters, and
$\operatorname{sgn}(\tau) \in \{\pm 1\} \subset k^\times$ is the sign
of the permutation $\tau \in \Sym_m$ (note that if $m>1$, then
$p>2$).  We leave to the reader the task of checking that the rules
above yield well-defined $G$-homomorphisms which determine sections
$\sigma_i$ to the maps $\pi_i$ for $i=1,2$.
\end{proof}

\subsection{Donkin pairs}
Let $H \subset G$ be a closed subgroup, and suppose that both $H$ and
$G$ are connected and reductive.  We choose Borel subgroups $B_H
\subset H$ and $B_G \subset G$ with $B_H \subset B_G$, and we choose
maximal tori $T_H \subset B_H$ and $T_G \subset B_G$ with $T_H \subset
T_G$.

Write $\wts G := X^*(T_G)$ for the weight lattice of $T_G$, and let
$\dom G \subset X_G$ denote the dominant weights (determined by the
choice of Borel subgroup $B_G$); similarly, write $\dom H \subset \wts
H := X^*(T_H)$. For $\lambda \in \dom G$ recall that $\ho{G}{\lambda}$
is the induced $G$-module with highest weight $\lambda$. Similarly, for
$\mu \in \dom H$ we write $\ho{H}{\mu}$ for the induced $H$-module
with highest weight $\mu$.

\begin{definition} One says that $(G,H)$ is a \emph{Donkin pair} if
 whenever $V$ is a $G$-module for which $V$ has a good filtration,
 then $\res^G_HV$ has a good filtration as $H$-module. One also says
 that $H$ is a \emph{good filtration subgroup} of $G$.
\end{definition}

Let us write $\goodwt(G,H)$ for the set of dominant weights $\lambda
\in \dom G$ for which $\res^G_H \ho{G}{\lambda}$ has a good filtration
as $H$-module.

\begin{stmt}
 $(G,H)$ is a Donkin pair if and only if $\lambda \in \goodwt(G,H)$
 for each $\lambda \in \dom G$.
\end{stmt}
\begin{proof} This follows from \ref{good-filtration-filtration}.
\end{proof}

An important example of Donkin pairs is given by the following result;
as for \ref{tensor-good-filtration}, this result was obtained
first by Donkin \cite{donkin} (under some mild assumptions)
and subsquently by Mathieu \cite{mathieu} in general.
\begin{stmt}
\label{Levi-Donkin-pair}
If $L \subset G$ is a Levi factor of a parabolic subgroup of $G$,
then $(G,L)$ is a Donkin pair.
\end{stmt}

\begin{remark}
\label{remark-tensor-via-diagonal}
One can understand the earlier result \ref{tensor-good-filtration}
using the notion of Donkin pairs. Namely, if $G$ is a connected and
reductive group, let $\Delta:G \to G \times G$ denote the diagonal
embedding. Then \ref{tensor-good-filtration} amounts to the assertion
that $(G\times G,\Delta(G))$ is a Donkin pair.
\end{remark}

\begin{stmt}
 \label{direct-factor-Donkin-subgroup}
 Let $G_1$ and $G_2$ be connected and reductive algebraic groups, and
 let $G = G_1 \times G_2$. 
 \begin{enumerate}[(a)]
 \item  $(G,G_1)$ and $(G,G_2)$ are Donkin pairs.
 \item If $H \subset G$ is a connected and reductive subgroup, write
   $H_i \subset G_i$ for the image $H_i = \pi_iH \subset G_i$ where
   $\pi_i:G \to G_i$ are the projection mappings.  If $(G_i,H_i)$ is
   a Donkin pair for $i=1,2$, then $(G,H)$ is a Donkin pair.
 \end{enumerate}
\end{stmt}

\begin{proof}
 (a) follows from \cite{donkin-normality}, Prop. 1.2(e).

 For (b), argue as in \cite{donkin} (3.4.6) to see that $(G =
 G_1 \times G_2,H_1 \times H_2)$ is a Donkin pair, so it suffices to
 see that $(H_1 \times H_2,H)$ is a Donkin pair. Now, the inclusion $H \hookrightarrow H_1 \times H_2$ factors as $$H \hookrightarrow H \times H \stackrel q \twoheadrightarrow H_1 \times H_2 \,,$$ where the first map is the diagonal inclusion and $q$ is the restriction of the projection $\pi_1 \times \pi_2$ to $H \times H$. Since the pullback by $q$ of an induced module for $H_1 \times H_2$ is an induced module for $H \times H$ (since $q$ is a surjection), the fact that $(H_1 \times H_2,H)$ is a Donkin pair now follows from \ref{tensor-good-filtration} as in Remark \ref{remark-tensor-via-diagonal}.
\end{proof}

\begin{stmt}
 \label{reduce-donkin-pair-to-simply-connected-case}
 Let $G$ be a connected and semisimple group, let $H$ be a connected
 and reductive subgroup of $G$, and let $\pi:G_{\sconn} \to G$ be the
 simply connected covering group. If $(G_{\sconn},\pi^{-1} H)$ is a
 Donkin pair, then $(G,H)$ is a Donkin pair.
\end{stmt}

\begin{proof}
 Argue via \cite{donkin}, 3.4.3.
\end{proof}

\subsection{Checking for a Donkin pair using finitely many dominant weights}
Let $G$ be a semisimple group, and fix a system of simple roots $S
\subset R \subset X=X^*(T)$. We have the following generalization of a result of Donkin found in \cite{donkin} Prop. 3.5.4, whose proof we have followed closely.

\begin{stmt}
\label{Donkin-pair-by-set-of-weights}
Suppose that $\lambda_1,\dots,\lambda_r \in \dom G$ have the property
\begin{equation*}
  (*) \quad \mu \in \dom G \quad \text{if and only if} 
  \quad \mu \in \sum_{i =1}^r \Z_{\ge 0} \lambda_i
\end{equation*}
for $\mu \in X$.  If $H$ is a reductive subgroup of $G$, then
$(G,H)$ is a Donkin pair if and only if $\lambda_i \in \goodwt(G,H)$
for $1 \le i \le r$.
\end{stmt}

\begin{proof}
This follows from the same technique as in the proof of
 \cite{donkin}, Prop. 3.5.4. We give a full proof here for completeness. Consider the partial order on $X$ given
 as follows: $\lambda \succ \mu$ if and only if $\lambda - \mu =
 \sum_{\alpha \in S} m_\alpha \alpha \in X \tensor_\Z \Q$ where
 $m_\alpha \in \Q$ and $m_\alpha \ge 0$ for all $\alpha$. Note that
 if $\lambda \ge \mu$ then $\lambda \succ \mu$.  It follows from
 \cite{humphreys-LA}, \S13, Exerc. 8 that any dominant weight
 $\lambda$ satisfies $\lambda \succ 0$. In view of the assumption
 $(*)$, for any $\mu \in \dom G$ with $\mu \ne 0$ we have $\mu \prec
 \mu - \lambda_j$ for some $1 \le j \le r$.

 We now give the proof.  If $(G,H)$ is a Donkin pair, it is of course
 immediate from definitions that each $\lambda_i \in \goodwt(G,H)$
 for $1 \le i \le r$.

 We now suppose that $\lambda_i \in \goodwt(G,H)$ for $1 \le i \le
 r$, and we must show that $(G,H)$ is a Donkin pair. Suppose on the
 contrary that $(G,H)$ is not a Donkin pair.  Then there is a weight
 $\mu \in X_G^+$ for which $\mu \not \in \goodwt(G,H)$.  We may and
 will suppose that $\mu \in \dom G$ is minimal with respect to the
 partial order $\prec$; i.e., we suppose that $\mu \notin
 \goodwt(G,H)$ and $\lambda \in \goodwt(G,H)$ for all $\lambda \in
 \dom G$ with $ \lambda \prec \mu $. Since $0 \in \goodwt(G,H)$ we have
 $\mu \ne 0$. By hypothesis, as noted above, we have $\mu
 \succ \mu - \lambda_j \in \dom G$ for some $1 \leq j \leq r$. Since
 $\lambda_j \in \goodwt(G,H)$ by assumption, we see that $\mu -
 \lambda_j \ne 0$.

 We now consider the $G$-module $T = \ho{G}{\mu - \lambda_j} \tensor
 \ho{G}{\lambda_j}$.  By the minimality of $\mu$, we have $\mu -
 \lambda_j \in \goodwt(G,H)$. Since we have also $\lambda_j \in
 \goodwt(G,H)$ by assumption, \ref{tensor-good-filtration} implies
 that $T$ has a good filtration as $G$-module and that $\res^G_H T$
 has a good filtration as an $H$-module.

According to \ref{tensor-induced-quotient}, there is an exact sequence
of $G$-modules
\begin{equation*}
  0 \to M \to T  \to
  \ho{G}{\mu} \to 0
\end{equation*}
for which $M$ has a good filtration as $G$-module. Moreover, since
the dimension of the $\mu$-weight space in $\ho{G}{\mu - \lambda_j}
\otimes \ho{G}{\lambda_j}$ is 1, it follows from
\cite{JRAG}, Prop. II.4.16 that in any good filtration of $T$,
there is precisely one layer isomorphic to $\ho{G}{\mu}$. In
particular, $\ho{G}{\mu}$ does not appear as a layer in the
$G$-module $M$.

Since any weight $\gamma$ of $T$ satisfies $\mu \ge \gamma$ and in
particular $\mu \succ \gamma$, it follows that $\res^G_H M$ has a
good filtration as $H$-module; cf.  \ref{good-filtration-filtration}.
It now follows from \cite{JRAG}, II.4.17 that $\res^G_H \ho{G}{\mu}$
has a good filtration as $H$-module, so that $\mu \in \goodwt(G,H)$,
contrary to assumption. This contradiction establishes that $(G,H)$
is a Donkin pair, as required.
\end{proof}

\begin{remark}
\label{simply-connected-weight-lattice-donkin-pair-criteria}
In particular, we obtain the following statement (\cite{donkin}, Prop. 3.5.4): If $G$ is semisimple and simply connected of rank $r$, then the
 assumption of \ref{Donkin-pair-by-set-of-weights} holds for the
 collection of weights $\lambda_i = \varpi_i$ for $1 \le i \le r$,
 where $\varpi_i$ is the $i$-th fundamental dominant weight.

On the other hand, suppose that $G$ is semisimple, that
$\pi:G_{\sconn} \to G$ is its simply connected covering group, and
that $H \subset G$ is a reductive subgroup. Write $\tilde H =
\pi^{-1}H \subset G_{\sconn}$. The formulation
\ref{Donkin-pair-by-set-of-weights} will allow us to check that
$(G,H)$ is a Donkin pair in situations when we are unable to determine
whether $(G_{\sconn},\tilde H)$ is a Donkin pair, cf \ref{odd-orthog-description}, \ref{even-orthog-description}, and \ref{check-Donkin-pair-on-exterior-powers} below. %Theorem \ref{classical-optimal-donkin-pair} and Remark \ref{spin unknown} below.
\end{remark}

\subsection{Optimal $\SL_2$-subgroups}
\label{optimal-SL2-subgroups}
In this section, let $H$ be a quasisimple group with root system $R$.
We want to consider primes which are \emph{good} for $G$ or
equivalently primes which are \emph{good} for $R$. Recall that bad (=not
good) primes are as follows: the prime $p=2$ is bad whenever $R \not =
A_r$, $p=3$ is bad if $R = G_2,F_4,E_r$, and $p=5$ is bad if $R=E_8$.
Finally, the prime $p$ is said to be \emph{very good} for $R$ if $p$
is good for $R$ and if $R = A_r$ then $p$ does not divide $r+1$.

The following result can be deduced as a consequence of Premet's proof
of the Bala-Carter Theorem which classifies the nilpotent $G$-orbits
in $\Lie(G)$:
\begin{stmt}
 \label{assoc-cochar}
 Let $H$ be a quasisimple group and suppose that the characteristic
 of $k$ is good for $H$. Let $G$ be a reductive group which is
 isomorphic to a Levi factor of a parabolic subgroup of $H$.  If $X
 \in \Lie(G)$ is nilpotent, then there is a cocharacter $\lambda:\Gm
 \to G$ such that
 \begin{enumerate}[(i)]
 \item  for each $t \in k^\times$, $\Ad(\lambda(t))X = t^2 X$.
 \item the image of $\lambda$ is contained in the derived group of $M
   = C_G(S)$ for some maximal torus $S$ of the group $C_G(X)$.
 \end{enumerate}
 If $\lambda,\lambda':\Gm \to G$ are two cocharacters satisfying (i)
 and (ii), there is a unique element $u \in R_u(C_G(X))$ such that
 $\lambda'(t) = u \lambda(t) u^{-1}$ for each $t \in k^\times$;
 i.e. $\lambda' = \Int(u) \circ \lambda$.
\end{stmt}

\begin{proof}
 \cite{mcninch-rat}, Proposition 18 shows how to deduce the
 existence of $\lambda$ from results of \cite{premet}. For the
 conjugacy assertion see \cite{mcninch-optimal}, Prop/Def 21.
\end{proof}

If $X \in \Lie(G)$ is nilpotent, we say that a cocharacter $\lambda$
satisfying (i) and (ii) of \ref{assoc-cochar} is \emph{associated
 with} $X$.

Let $X_0 = \begin{pmatrix}
 0 & 1 \\
 0 & 0
\end{pmatrix}
\in \Lie(\SL_2)$.  One says that a homomorphism $\phi:\SL_2 \to G$ is
\emph{optimal} if the cocharacter of $G$ determined by restriction of
$\phi$ to the diagonal torus of $\SL_2$ is associated with the
nilpotent element $X = d\phi(X_0)$.

\begin{stmt}
 \label{optimal-central-isogeny}
 Let $\pi:G_1 \to G$ be a central isogeny and let $\phi:S \to G_1$ be
 a homomorphism of algebraic groups. Then $\phi$ is optimal if and
 only if $\pi \circ \phi$ is optimal.
\end{stmt}

\begin{proof}
 The assertion follows from \cite{mcninch-rat}, Lemma 14.
\end{proof}

\begin{theorem}[\cite{mcninch-optimal}, Theorem 44] \label{th:mcninch-optimal}
 Let $X \in \Lie(G)$ be nilpotent and suppose that $X^{[p]} = 0$.
 There exists an optimal homomorphism $\phi:\SL_2 \to G$ with $X =
 d\phi (X_0 )$.  If $\phi'$ is another optimal homomorphism $\SL_2
 \to G$ with $X = d\phi(X_0)$, then $\phi$ and $\phi'$ are conjugate
 by a unique element of $R_uC_G(X)$.
\end{theorem}

The image $S = \operatorname{image}(\phi)$ of an optimal
$\SL_2$-homomorphism will be called an \emph{optimal} $\SL_2$-subgroup
(although of course it may be that $S \simeq \operatorname{PSL}_2$ is
the adjoint group).

One says that an optimal $\SL_2$-subgroup $S$ of $G$ is
\emph{distinguished} just in case a nilpotent element $0 \ne X \in
\Lie(S) \subseteq \Lie(G)$ is distinguished.

\begin{stmt}
 \label{distinguished-optimal-sl2s}
 If $S$ is an optimal $\SL_2$-subgroup of a quasisimple group $G$,
 then $S$ is distinguished if and only if a maximal torus of the
 centralizer $C_G(S)$ is trivial.
\end{stmt}

\begin{proof}
First suppose that $S$ is distinguished.  Let $0 \ne X \in
\Lie(S)$. Since $X$ is distinguished, by definition a maximal torus
$T$ of $C_G(X)$ is central in $G$; since $G$ is quaisimple, $T=1$.
It follows from \cite{mcninch-optimal}, Cor. 43 that $C_G(S)
\subset C_G(X)$ so indeed a maximal torus of $C_G(S)$ is trivial.

Conversely, suppose that a maximal torus of $C_G(S)$ is trivial and
let $T \subset C_G(X)$ be a maximal torus. To show that $X$ is
distinguished, we must argue that $T$ is trivial.  Write $\lambda$
for the cocharacter of $G$ obtained by the inclusion of the maximal
torus of $S$.  By the conjugacy of maximal tori in $C_G(X)^0$, we may
suppose that $T$ is centralized by the image of $\lambda$.  But then
\cite{mcninch-optimal}, Cor. 43 shows that $T$ centralizes $S$, so
indeed $T =1$.
\end{proof}

\section{The groups of interest} \label{sec:the groups of interest}

\subsection{``Classical'' groups}
\label{classical-groups}
Let $V$ be a vector space, and let $\beta$ be a non-degenerate
bilinear form on $V$.  Write
\begin{equation*}
\Omega(V) = \Omega(V,\beta) = \Stab_{\GL(V)}(\beta)^0
\end{equation*}
for the identity component of the stabilizer of $\beta$ in
$\GL(V)$. There are two cases of interest to us:
\begin{enumerate}[(a)]
%\item If $\beta = 0$, then $\Omega(V) = \GL(V)$.
\item If $\beta$ is alternating, then $\Omega(V) = \Symp(V)$ is a
 symplectic group.
\item If $\beta$ is symmetric, then $\Omega(V) = \SOrth(V)$ is a
 special orthogonal group, when $p > 2$.
\end{enumerate}
The definition of the special orthogonal group requires more care when
$p=2$, and we ignore this issue.

In this paper, a \emph{classical group} will mean a group of the form
\begin{equation*}
 G = \SL(V) 
\quad \text{or} \quad G=\Omega(V,\beta)
\text{ in case (a) or (b) for $p > 2$.}
\end{equation*}

\subsection{Good filtration subgroups of classical groups}
\label{classical groups and their subgroups}
Let $G$ be a classical group as in \S \ref{classical-groups}.  We say
that $V$ is the \emph{natural representation} of the classical group
$G$. We will often just write $\Omega(V)$ instead of $\Omega(V,
\beta)$.

\begin{stmt}
 Let $G$ be a classical group with natural representation $V$. Then $V \simeq \ho{G}{\lambda}$ for some dominant weight $\lambda$, and $V$ is simple.
\end{stmt}

\begin{proof}
 See, for example, \cite{JRAG}, \S II.2.16, 17, 18. The simplicity of
 $V$ when $G = \Omega(V,\beta)$ for $\beta$ symmetric depends in general on the assumption
 that $p \ne 2$.
\end{proof}

Recall that the fundamental dominant weights of $G$ are denoted by
$\varpi_1, \ldots, \varpi_r$ where $r$ is the rank of the semisimple
group $G$; in general $\varpi_i \in X^*(T) \tensor \Q$. In what
follows, we abbreviate $X := X^*(T)$.

\begin{stmt}
\label{SLV-description}
Let $G = \SL(V)$, where $\dim V = r+1$. Then $G$ is a simply
connected group with root system of type $A_r$, and we have
\begin{equation*}
 \ho{ \SL(V) }{\varpi_i} \simeq \exterior^i V \quad \text{for} 
 \quad 1 \le i \le r.
\end{equation*}
\end{stmt}

\begin{proof}
 \cite{involutions}, Theorem (25.9) identifies the root system of
 $\SL(V)$. Apply e.g. \cite{donkin}(4.1.1) for the assertion about
 exterior powers.
\end{proof}

\begin{stmt}
\label{SpV-description}
Assume that $\beta$ is alternating and non-degenerate and that $\dim
 V = 2r$. Then $\Omega(V)=\Symp(V)$ is a simply connected quasi-simple
group with root system of type $C_r$ and the exterior powers
$\exterior^j V$ have a good filtration as $\Omega(V)$-module. When $p > 2$, for
each $1 \le i \le r$ there is an exact sequence
\begin{equation*}
  0 \to \exterior^{i-2} V \to \exterior^i V \to \ho{\Symp(V)}{\varpi_i} \to 0
\end{equation*}
of $\Omega(V)$-modules (where $\exterior^{-1} V = 0$ and
$\exterior^0 V = k$).
\end{stmt}

\begin{proof}
It follows from \cite{involutions}, Theorem 25.11 that $\Omega(V)
= \Symp(V)$ is quasisimple of type $C_r$. For $p>2$ the results
on exterior powers of $V$ follow from \cite{andersen-jantzen} 4.9. The fact that the modules $\exterior^j V$ have a good filtration when $p=2$ follows from \cite{Do94}, Appendix A.
\end{proof}

%\begin{remark}
%According to \cite{Do94}, Appendix A the exterior powers $\exterior^i V$ in the above statement have a good filtration as an $\Symp(V)$-module in characteristic 2 as well.
%\end{remark}

\begin{stmt}
\label{odd-orthog-description}
Assume that $\beta$ is symmetric and non-degenerate, that $p>2$ and
that $\dim V = 2r+1$. Then
\begin{enumerate}[(a)]
\item $\Omega(V) = \SOrth(V)$ is a quasisimple group of type $B_r$.
\item $\varpi_i \in X_{\SOrth(V)}$ for $1 \le i \le r-1$, $2\varpi_r \in X_{\SOrth(V)}$
and for any $\lambda \in X_{\SOrth(V)}$, we have 
\begin{equation*}
\lambda \in X^+_{\SOrth(V)} \quad \text{if and only if} \quad
\lambda \in \sum_{i=1}^{r-1} \Z_{\ge 0} \, \varpi_i + \Z_{\ge 0} \, 2\varpi_r.
  \end{equation*}
\item For $1 \le i \le r-1$, $\ho{\SOrth(V)}{\varpi_i} \simeq
  \exterior^i V$, and $\ho{\SOrth(V)}{2\varpi_r} \simeq \exterior^r
  V$.
\end{enumerate}
\end{stmt}

\begin{proof}
(a) follows from \cite{involutions}, Theorem 25.10.  Moreover, by
\emph{loc. cit.} one knows that $\Omega(V)$ is not simply connected.
Now (b) follows from the description in \cite{bourbaki}, Plate II.
Finally, (c) is verified in \cite{andersen-jantzen}, 4.9.
\end{proof}

\begin{stmt}
\label{even-orthog-description}
Assume that $\beta$ is symmetric and non-degenerate, that $p>2$, and
that $\dim V = 2r$. Then
\begin{enumerate}[(a)]
\item $\Omega(V) = \SOrth(V)$ is a quasisimple group of type $D_r$.
\item $\varpi_i \in X_{\SOrth(V)}$ for $1 \le i \le r-2$, $\varpi_{r-1} +
 \varpi_r, 2\varpi_{r-1}, 2\varpi_r \in X_{\SOrth(V)}$ and for any $\lambda \in
 X_{\SOrth(V)}$, we have 
\begin{equation*}
 \lambda \in X^+_{\SOrth(V)} \quad \text{if and only if} \quad
 \lambda \in \sum_{i=1}^{r-2} \Z_{\ge 0} \, \varpi_i + \Z_{\ge 0}(\varpi_{r-1} + \varpi_r) + 
 \Z_{\ge 0} \, 2\varpi_{r-1} + \Z_{\ge 0} \, 2\varpi_r.
  \end{equation*}
\item For $1 \le i \le r-2$, $\ho{\SOrth(V)}{\varpi_i} \simeq
  \exterior^i V$. Moreover,
  \begin{equation*}
    \exterior^{r-1}V \simeq \ho{\SOrth(V)}{\varpi_{r-1} + \varpi_r}
    \quad \text{and} \quad
    \exterior^r V \simeq \ho{\SOrth(V)}{2\varpi_r} \oplus \ho{\SOrth(V)}{2\varpi_{r-1}}.
  \end{equation*}
\end{enumerate}
\end{stmt}

\begin{proof}
(a) follows from \cite{involutions}, Theorem 25.12.  Moreover, by
\emph{loc. cit.} one knows that $\Omega(V)$ is neither simply
connected nor adjoint.  In this case, there are three groups in the
isogeny class which are neither adjoint nor simply connected;
$\Omega(V)$ is characterized by the fact that $X$ contains neither
$\varpi_{r-1}$ nor $\varpi_r$.

Now (b) follows from the description in \cite{bourbaki}, Plate IV.
The assertions in (c) about $\exterior^i V$ for $i < r$ are verified
in \cite{andersen-jantzen}, 4.9. The assertion about $\exterior^r
V$ is proved in \cite{mcninch-exterior}, Remark 3.4.
\end{proof}

We conclude this discussion with the following result, which is similar
to the methods used in \cite{donkin} and \cite{brundan}.
\begin{stmt}
\label{check-Donkin-pair-on-exterior-powers}
Let $G$ be a classical group with natural representation $V$ and
assume that $p > 2$ if $G \neq \SL(V)$. Let $H \subset G$ be a
connected and reductive subgroup.  Then $(G, H)$ is a Donkin pair if
and only if the exterior algebra $\exterior^\bullet V$ has a good
filtration as an $H$-module.
\end{stmt}

\begin{proof}
 Together with \ref{direct-summand-good-filt}, the descriptions found
 in \ref{SLV-description}, \ref{SpV-description},
 \ref{odd-orthog-description}, and \ref{even-orthog-description}
 show that the exterior algebra $\exterior^\bullet V$ has a good
 filtration as $G$-module.  Thus if $(G,H)$ is a Donkin pair then $\exterior^\bullet V$ has a good
 filtration as $H$-module.

 Conversely, suppose that $\exterior^\bullet V$ has a good filtration as
 $H$-module.  Of course, according to \ref{direct-summand-good-filt},
 also each exterior power $\exterior^i V$ has a good filtration as
 $H$-module.

 To show that $(G,H)$ is a Donkin pair, we are going to apply
 \ref{Donkin-pair-by-set-of-weights}. We verify for each classical
 group $G$ that there is a set of dominant weights
 $\lambda_1,\dots,\lambda_t \in \goodwt(G,H)$ for which $\lambda \in
 X$ satisfies
 \begin{equation*}
   (*) \quad \lambda \in \dom G \quad \text{if and only if} \quad
   \lambda \in \sum_{i=1}^t \Z_{\ge 0} \lambda_i \,.
 \end{equation*}

 When $G = \SL(V)$, with $\dim V = r+1$, we take $t=r$ and we let
 $\lambda_i = \varpi_i$ for $1 \le i \le r$. According to
 \ref{SLV-description}
 \begin{equation*}
   \ho{\SL(V)}{\lambda_i} = \ho{\SL(V)}{\varpi_i} \simeq \exterior^i V
 \end{equation*}
 which has a good $H$-module filtration by assumption.

 Now let $\beta$ be non-degenerate and alternating and let $\dim V =
 2r$. Then \ref{SpV-description} shows that $G=\Omega(V)$ is again
 simply connected; again we take $t=r$ and we let $\lambda_i =
 \varpi_i$ for $1 \le i \le r$. Applying \ref{SpV-description}, we
 have for each $1 \le i \le r$ an exact sequence of $H$-modules
\begin{equation*}
  0 \to \exterior^{i-2} V \to \exterior^i V \to 
  \res^{\Omega(V)}_H \ho{\Omega(V)}{\varpi_i} \to 0.
\end{equation*}
Applying \ref{direct-summand-good-filt}(a) together with our
assumption, we conclude that
$\res^{\Omega(V)}_H\ho{\Omega(V)}{\varpi_i}$ has a good $H$-module
filtration for all $1 \leq i \leq r$.

Now let $\beta$ be symmetric and non-degenerate and let $\dim
V = 2r+1$. We take $t=r$ and set $\lambda_i = \varpi_i$ for $i<r$
and $\lambda_r = 2\varpi_r$. Then \ref{odd-orthog-description} shows
that the $\lambda_i$ satisfy $(*)$. Moreover, the same result shows
that
\begin{equation*}
  \res^{\Omega(V)}_H \ho{\Omega(V)}{\lambda_i} \simeq \exterior^i V
\end{equation*}
has a good $H$-module filtration by assumption.

Finally, let $\beta$ be symmetric and non-degenerate and let $\dim V
= 2r$. We take $t = r+1$, we set $\lambda_i = \varpi_i$ for $i <
r-1$, and we set $\lambda_{r-1} = \varpi_r + \varpi_{r+1}$,
$\lambda_r = 2\varpi_r$, and $\lambda_{r+1} = 2\varpi_{r-1}$.  It
follows from \ref{even-orthog-description} that $(*)$ holds.  Now,
the same result shows that the $H$-modules
\begin{equation*}
  \res^{\Omega(V)}_H \ho{\Omega(V)}{\lambda_i} \simeq \exterior^i V 
  \quad (1 \le i \le r-2), \quad \text{and} \quad
  \res^{\Omega(V)}_H \ho{\Omega(V)}{\lambda_{r-1}} \simeq \exterior^{r-1} V
\end{equation*}
have a good $H$-filtration by assumption. Finally, since by
\ref{even-orthog-description} we have 
\begin{equation*}
  \exterior^r V \simeq
  \res^{\Omega(V)}_H \ho{\Omega(V)}{2\lambda_r} \oplus
  \res^{\Omega(V)}_H \ho{\Omega(V)}{2\lambda_{r-1}}
\end{equation*}
we conclude via \ref{direct-summand-good-filt}(b) that
$\res^{\Omega(V)}_H \ho{\Omega(V)}{2\lambda_j}$ has a good
filtration as $H$-module for $j=r,r+1$. This completes the proof.
\end{proof}

Recall from \ref{SLV-description} and\ref{SpV-description} that
$\SL(V)$ is simply connected, and $\Omega(V)$ is simply connected if
$\beta$ is alternating. In order to study good filtration subgroups of
the simply connected covers of orthogonal groups, one may use the
following:

\begin{stmt}
 \label{donkin-subgroups-of-spin}
 Suppose that $p > 2$ and let $G = \Omega(V,\beta)$ for a
 non-degenerate symmetric form $\beta$.  Let $\pi:G_{\sconn} \to G$ be
 the \emph{simply connected} covering group of $G$, and let $H
 \subset G_{\sconn}$ be a connected and reductive subgroup.  Then
 $(G_{\sconn},H)$ is a Donkin pair if and only if $(G,\pi H)$ is a
 Donkin pair and one of the following conditions holds:
 \begin{enumerate}[(a)]
 \item $\dim V = 2r+1$ is odd and $\res^{G_{\sconn}}_{H}
   \ho{G_{\sconn}}{\varpi_r}$ has a good filtration, with numbering
   of the fundamental dominant weights as in
   \ref{odd-orthog-description}, or
 \item $\dim V = 2r$ is even and $\res^{G_{\sconn}}_{H}
   \ho{G_{\sconn}}{\varpi_i}$ has a good filtration for $i=r-1,r$,
   with numbering of the fundamental dominant weights as in
   \ref{even-orthog-description}.
 \end{enumerate}
\end{stmt}

\begin{proof}
 We apply \ref{Donkin-pair-by-set-of-weights} using the set of
 fundamental dominant weights as in Remark
 \ref{simply-connected-weight-lattice-donkin-pair-criteria}.  The
 result now follows from the descriptions found in \ref{odd-orthog-description} and
 \ref{even-orthog-description} together with \ref{check-Donkin-pair-on-exterior-powers}.
\end{proof}

\subsection{Some reductive subgroups of a classical group}
Let $V$ be a finite dimensional $k$-vector space, let $\beta$ be a
non-degenerate alternating or symmetric bilinear form on $V$ and
suppose that $p > 2$. As above put $\Omega(V) = \Omega(V,\beta)$.

Let $W \subset V$ be a linear subspace of $V$. We say that $W$ is
\emph{non-degenerate} if the restriction of $\beta$ to $W$ is
non-degenerate, and we say that $W$ is \emph{isotropic} if the
restriction of $\beta$ to $W$ is identically zero.

\begin{stmt}
 Let $W_1,\dots,W_r \subset V$ be non-degenerate subspaces with
 $\beta(W_i,W_j)=0$ for $i \ne j$. Then
 \begin{equation*}
   (\Omega(V),\Omega(W_1) \times \cdots \times \Omega(W_r))
 \end{equation*}
 is a Donkin pair.
\end{stmt}

\begin{proof}
 \cite{brundan}, Prop. 3.3; here we are using that $p > 2$.
\end{proof}

\begin{stmt}
 \label{levi-of-classical}
 Let $P \subset \Omega(V)$ be a parabolic subgroup and let $L \subset P$
 be a Levi factor. Then there is a non-degenerate subspace $W \subset
 V$ and an isotropic subspace $U \subset V$ such that $L \simeq L_1
 \times L_2$ where $L_1$ is equal to $\Omega(W)$ and $L_2$ is a Levi
 factor of a parabolic subgroup of $\SL(U)$.
\end{stmt}

\begin{proof}
 The result follows from the well-known observation that a parabolic
 subgroup of $\Omega$ is the stabilizer of a flag of isotropic
 subspaces of $V$.
\end{proof}

\subsection{Some modules for $\SL_2$ having a good filtration}

We write $S$ for the simple algebraic group $\SL_2$.

Let $\varpi \in \dom S$ denote the fundamental weight for some choice of maximal torus of Borel subgroup of $S$. For any integer $n \ge 0$ set $\ho{S}{n}:= \ho S{n\varpi}$, the unique induced $S$-module having dimension $n + 1$.

\begin{stmt}
 \label{SL2-simple-hos}
 The $S$-module $\ho{S}{n}$ is simple if and only if $n < p$. 
\end{stmt}

\begin{proof}
 This is a consequence of the Linkage Principle \cite{JRAG}, \S II.6.
\end{proof}

A semisimple $S$-module $V$ will be called \emph{restricted} (or
\emph{restricted semisimple}) if $\dim_k V^S = \dim_k V^{\Lie(S)}$.
\begin{stmt}
 \label{restricted-ss-for-sl2}
 The following are equivalent for an $S$-module $V$.
 \begin{enumerate}[(a)]
 \item   $V$ is a restricted semisimple $S$-module 
 \item There is an isomorphism $V \simeq
   \bigoplus_{i \in I} \ho{S}{n_i}$ where $0 \le n_i < p$ for all $i
   \in I$
 \item If $V_{m\varpi} \ne 0$ then $m < p$.
 \end{enumerate}
\end{stmt}

\begin{proof}
 First, suppose that $L$ is a simple $S$-module. If the highest
 weight $n\omega$ of $L$ satisfies $n<p$, it follows from
 \cite{JRAG}, \S II.3.15 that $L$ is simple as a module for $\Lie(S)$
 as well.  For any simple $S$-module $L$, it now follows from
 Steinberg's Tensor Product Theorem \cite{JRAG}, II.3.17 that $L$ is
 semisimple as a $\Lie(S)$-module. If $n\omega$ is the highest weight
 of $L$, Steinberg's Tensor Product Theorem yields $L^S =
 L^{\Lie(S)}$ if and only $n<p$.  It is then clear that (a) implies
 (b). If (b) holds, evidently (c) holds as well.

 Finally, if (c) holds, it follows from \ref{SL2-simple-hos} that
 the simple submodules of $V$ all have the form $L(n) = \ho{S}{n}$
 for $n<p$. The Linkage Principle \cite{JRAG}, II.6 shows that
 $\Ext^1_S(L(n),L(m)) = 0$ whenever $0 \le n,m < p$ so indeed $V$ is
 semisimple and (a) follows.
\end{proof}

\begin{stmt}
 \label{exterior-of-restricted-semisimple-for-SL2} 
 Let $E$ be a finite dimensional restricted semisimple $S$-module.
 Then $\exterior^\bullet E$ has a good filtration as an $S$-module.
\end{stmt}

\begin{proof}
 By hypothesis and \ref{restricted-ss-for-sl2}, we may write $E =
 \bigoplus_{i=1}^r E_i$ where $E_i \simeq \ho{S}{n_i}$ with $0 \le
 n_i < p$ $i=1,\dots,r$.  Then $\dim E_i \le n_i +1 \le p$ for each
 $i$.

 We may evidently view $\exterior^\bullet E_i$ as an $S$-submodule of
 $\exterior^\bullet E$ for each $i$, and multiplication in the algebra
 $\exterior^\bullet E$ defines a $S$-module isomorphism
  \begin{equation*}
    \exterior^\bullet E_1 \tensor \exterior^\bullet E_2 \tensor  
    \cdots \tensor   \exterior^\bullet E_r
    \xrightarrow{\sim}
    \exterior^\bullet E
  \end{equation*}
  In view of \ref{tensor-good-filtration}, it suffices to prove the
  claim when $E = \ho{S}{n}$ for some $n < p$.

  Now, $\exterior^\bullet E = \bigoplus_{i=0}^{n+1} \exterior^i E$ as
  $S$-modules, so it is enough to see that $\exterior^i E$ has a good
  filtration as $S$-module for each $0 \le i \le n+1$.  Since $E$ has
  a good filtration as $S$-module, \ref{multilinear} shows that
  $\exterior^i E$ has a good filtration as $S$-module for each $i <
  p$. This completes the proof if $n < p-1$ since $\dim \ho S n = n + 1$.  If $n = p-1$, it only
  remains to note that $\dim E = p$ so that $\exterior^p E = k =
  \ho{S}{0}$ has a good filtration as $S$-module.
\end{proof}

\subsection{Groups of exceptional type} 
\label{sec:the exceptional case}

Let $G$ be a \emph{simply connected}, quasisimple algebraic group
whose root system $R$ is of exceptional type of rank $r$: i.e. $R$ is
one of $G_2,F_4,E_6,E_7,$ or $E_8$. We always number the simple roots
of $R$ according to the tables found in \cite{bourbaki}, Plate I-IX.

We begin with the following observations

\begin{stmt}
 \label{Levi factors of exceptional-type groups}
 Let $L$ be a Levi factor of a parabolic subgroup of $G$. Then $L$ is
 isomorphic to a direct product $L \simeq L_1 \times \cdots \times
 L_t$ where for each $1 \le i \le t$, one of the following holds:
 \begin{enumerate}[(a)]
 \item $L_i$ is a torus, or
 \item $L_i$ a simply connected quasisimple group of exceptional
   type, or
 \item $L_i$ is a simply connected quasisimple group of type $A_r$
   with $r \le 7$, or
 \item $G$ has type $E_d$ for $d=6,7,$ or $8$ and $L_i$ is a simply
   connected quasisimple group of type $D_r$ with $r \le d - 1$, or
 \item $G$ has type $F_4$ and $L_i$ is a simply connected quasisimple
   group of type $B_2 = C_2$, $B_3$ or $C_3$.
 \end{enumerate}
\end{stmt}

\begin{proof}
 Since $G$ is simply connected, $L$ is also simply connected, so that
 $L$ is isomorphic to the product of its connected center and its
 simply connected derived group $L'$. The result now follows by
 inspection of the Dynkin diagram of $G$.
\end{proof}

\subsection{Good filtration subgroups of a group of exceptional type}

Let $W_p$ be the \emph{affine Weyl group} associated with $G$
\cite{JRAG}, \S II.6. We consider the so-called ``dot-action'' of
$W_p$ on the weight lattice $X = X^*(T)$: for $w \in W_p$ and $\mu \in
X$ we have $w \bullet \mu = w(\mu - \rho) + \rho$ where $\rho =
\frac{1}{2}\sum_{\alpha >0} \alpha$ is the half-sum of positive
roots.

For a reductive subgroup $H \subset G$, recall that $\goodwt(G,H)$
denotes the set of dominant weights for $G$ for which  the $H$-module
$\res^G_H\ho{G}{\lambda}$ has a good filtration. Recall also that for a $G$-module $V$ with a good filtration, we write $\supp(V)$ for the set of $\lambda \in \dom G$ for which
$\ho{G}{\lambda}$ occurs as a layer in a good filtration of $V$. We first recall the following fact from \cite{vdK}.

\begin{stmt} \label{vdk-lemma} Let $M$ be a $G$-module with good
 filtration for which $\res^G_H M$ has a good filtration as
 $H$-module.  Let $\lambda \in \supp(M)$.  Suppose that one of the
 following holds for each $\mu \in \supp(M)$ with $\mu \ne \lambda$:
\begin{enumerate}[(a)]
\item $\mu < \lambda$ and $\mu \in \goodwt(G,H)$, or
\item $\mu \not \in W_p \bullet \lambda$.
\end{enumerate}
Then $\lambda \in \goodwt(G,H)$.
\end{stmt}

\begin{proof}
This is Lemma 6.3, \cite{vdK}.
\end{proof}

For any $G$-module $M$, let $\chi(M)$ denote the character of $M$; see
\cite{JRAG}, \S I.2.11 and II.5. For $\lambda \in \dom G$ set
$\chi(\lambda) := \chi(\ho{G}{\lambda})$.

\begin{stmt}
 \label{good-filtration-layers-via-character}
 The characters $\chi(\lambda)$ for $\lambda \in \dom G$ form a
 $\Z$-basis of $\Z[T]^W$. In particular, if the finite dimensional
 $G$-module $M$ has a good filtration, then $\chi(M) = \sum_{\lambda
   \in \dom G} n_\lambda \chi(\lambda)$ where $n_\lambda \in \Z_{\ge
   0}$ is equal to the number of layers in a good filtration of $M$
 which are isomorphic to $\ho{G}{\lambda}$.
\end{stmt}

\begin{proof}
 This follows from \cite{JRAG}, Lemma and Remark II.5.8.
\end{proof}

The following result may be viewed as a sharpened version of
\ref{Donkin-pair-by-set-of-weights} valid for exceptional groups.

\begin{theorem} 
\label{th:exceptional group Donkin pairs}
Let $G$ be a simply connected quasisimple group of rank $r$ with
exceptional type root system, and let $H \subset G$ be a reductive
subgroup of $G$.
\begin{enumerate}[(i)]
\item If $G$ is of type $F_4$, $p \geq 5$, and $\varpi_4 \in
  \goodwt(G,H) $, then $(G,H)$ is a Donkin pair.
\item If $G$ is of type $E_6$, $p \geq 5$, and $\varpi_1,
  \varpi_6 \in \goodwt(G,H) $, then $(G,H)$ is a Donkin pair.
\item If $G$ is of type $E_7$, $p \notin \{ 2, 5, 7 \}$, and $\varpi_1, 
  \varpi_7 \in \goodwt(G,H) $, then $(G,H)$ is a Donkin pair.
\item If $G$ is of type $E_8$, $p \geq 7$, and $\varpi_1,
  \varpi_8 \in \goodwt(G,H) $, then $(G,H)$ is a Donkin pair.
\end{enumerate}
\end{theorem}

\begin{proof}
This theorem is proved using case-by-case computations with the computer program LiE \cite{lie}. Using \ref{simply-connected-weight-lattice-donkin-pair-criteria}, our goal is to show that $\varpi_i \in \goodwt(G,H)$ for $1 \le i \le r$.  In each of the cases (i)--(iv), the hypothesis gives a set of fundamental weights known to be in $\goodwt(G,H)$.

Next, using induced modules $\ho{G}{\lambda}$ for weights $\lambda$ known to be in $\goodwt(G,H)$, we construct certain modules $F$ for which $F$ has a good filtration as
 $G$-module and $\res^G_H F$ has a good filtration as $H$-module;
 that $F$ and $\res^G_H F$ have the required filtrations will in each
 case be clear from either \ref{tensor-good-filtration} or
 \ref{multilinear}. We then use the character $\chi(F)$ together with
 \ref{good-filtration-layers-via-character} to compute $\supp (F)$; the computation
 of the character of $F$ is achieved in some cases using \cite{lie}.

 We now apply \ref{vdk-lemma} to $F$ to obtain more weights known to lie in $\goodwt(G,H)$ and then we repeat the procedure described above. The proof will be complete once we know that $\varpi_i \in \goodwt(G,H)$ for all $i$.

 We describe details when $G$ has type $E_7$. In this case our initial
 assumption is that $\varpi_1,\varpi_7 \in  \goodwt(G,H)$.  Recall we assume that
 $p \not \in \{ 2, 5, 7 \} $.

 To argue that $\varpi_3 \in \goodwt(G,H)$ we set $F = \exterior^2
 \ho{G}{\varpi_1}$. Since $\chi( F ) = \chi( \varpi_1 ) + \chi(
 \varpi_3 ) $, since $\varpi_1 < \varpi_3 $, and since $\varpi_1 \in
 \goodwt(G,H)$, \ref{vdk-lemma} shows that $\varpi_3 \in \goodwt(G,H)$.

 Next, we set $F = \exterior^2 \ho G {\varpi_7}$. Since $\chi(F) =
 \chi( 0 ) + \chi( \varpi_6 )$, \ref{vdk-lemma} implies that
 $\varpi_6 \in \goodwt(G,H)$.

 Now set $F = \ho G { \varpi_1 } \otimes \ho G { \varpi_7 }$. The
 character of $F$ is given by $\chi(F) = \chi( \varpi_2 ) + \chi(
 \varpi_7 ) + \chi( \varpi_1 + \varpi_7 ) $. We have $ \varpi_7 <
 \varpi_2 < \varpi_1 + \varpi_7$, and $\varpi_2$ is not in the
 $W_p$-orbit of $\varpi_1 + \varpi_7$. Since $\varpi_1, \varpi_7 \in
 \goodwt(G,H)$, \ref{vdk-lemma} implies that $\varpi_2 \in \goodwt(G,H)$.
 
Also note that this implies $\varpi_1 + \varpi_7 \in \goodwt(G,H)$, a fact we will use below when analyzing $\ho G { \varpi_1 } \otimes \ho G { \varpi_2 }$.

 The dominant weights $\mu$ for which $\mu \le \varpi_2 + \varpi_7$
 are precisely $\varpi_2 + \varpi_7, \varpi_1, \varpi_3, \varpi_6,$ and $ 2
 \varpi_7$. Of these, we already know $\varpi_1,\varpi_3,\varpi_6
 \in \goodwt(G,H)$. Setting $F = Sym^2 \ho{G}{\varpi_7}$ it follows from
 \ref{vdk-lemma} that $2\varpi_7 \in \goodwt(G,H)$. A similar argument shows
 that $2\varpi_1 \in \goodwt(G,H)$. Now set $F = \ho{G}{\varpi_2} \tensor
 \ho{G}{\varpi_7}$; for each dominant weight $\mu \ne \varpi_2 +
 \varpi_7$ of $F$ we have $\mu \in \goodwt(G,H)$ so that by \ref{vdk-lemma} we have
 $\varpi_2 + \varpi_7 \in \goodwt(G,H)$.

 Let $F = \exterior^2 \ho G {\varpi_2}$ and observe that $\chi(F) =
 \chi(0) + \chi( \varpi_4 ) + \chi( \varpi_6 ) + \chi(\varpi_2 +
 \varpi_7) + \chi( 2 \varpi_1 ) $. Let $\mu$ be a weight for which
 $\ho{G}{\mu}$ is isomorphic to a layer in a good filtration of $F$;
 then $\mu \leq \varpi_4$.  If $\mu \ne \varpi_4$ we already
 know that $\mu \in \goodwt(G,H)$. Thus by \ref{vdk-lemma} we have $\varpi_4 \in \goodwt(G,H)$.

 Finally, set $F = \ho G { \varpi_1 } \otimes \ho G { \varpi_2
 }$. Then $\chi(F) = \chi( \varpi_2 ) + \chi( \varpi_5 ) + \chi(
 \varpi_7 ) + \chi(\varpi_1 + \varpi_7) + \chi(\varpi_1 + \varpi_2)
 $. Thus, if $\mu$ is a weight for which $\ho{G}{\mu}$ is isomorphic
 to a layer of a good filtration of $F$, then $\mu \leq \varpi_5$
 if $\mu \ne \varpi_1 + \varpi_2$.  Now, $\varpi_1 + \varpi_2$ is not
 in the $W_p$-orbit of $\varpi_5$. Since $\varpi_2,\varpi_7,\varpi_1
 +\varpi_7 \in \goodwt(G,H)$, we conclude by \ref{vdk-lemma} that $\varpi_5 \in \goodwt(G,H)$, and we now have that $\varpi_i \in \goodwt(G,H)$ for all $ 1 \leq i \leq 7 $. Hence the proof is completed for a group of type $E_7$.

Groups of type $F_4$, $E_6$ and $E_8$ are handled in a similar manner; we omit the details.
\end{proof}

\begin{remark} \label{rem:lie and Wp-orbits} We used the following algorithm in LiE to determine, for two weights $\mu$ and $\nu$, the primes $p$ for which these two weights are conjugate under $W_p$.

Let $\Lambda_R$ denote the root lattice of $G$. First, we check if $\mu - \nu$ is in $\Lambda_R$; if not, the two weights are clearly not $W_p$-conjugate. Next, note that $\mu$ is $W_p$-conjugate to $\nu$ if and only if $w \bullet \mu - \nu \in p \Lambda_R$ for some $w \in W$, so we use LiE to compute the finite set
\begin{equation*}
\{  w \bullet \mu - \nu : w \in W  \}
\end{equation*}
and determine for which primes $p$ the elements of this set are in $p \Lambda_R$.
\end{remark}

\section{Main results} \label{sec:main results}

\subsection{Optimal $\SL_2$-subgroups of classical groups}
\label{sub:optimal sl2 subgroups of classical groups}

Let $G$ be a classical group with natural representation $V$ as
in \S\ref{classical-groups}. Recall that $p>2$ if $G \neq \SL(V)$. Also recall the definition of optimal $SL_2$ subgroups from \S\ref{optimal-SL2-subgroups}.

\begin{stmt}
\label{optimal-for-classical}
Let $\phi:\SL_2 \to G$ be a homomorphism of algebraic groups. Then
$\phi$ is optimal if and only if $V$ affords a restricted semisimple
module for $\SL_2$.
\end{stmt}

\begin{proof}
 This is observed in \cite{mcninch-testerman-cr}, Remark 18 when $G
 = \SL(V)$, but the argument given there is valid for any $\Omega(V)$
 as well.
\end{proof}

\begin{theorem}
\label{classical-optimal-donkin-pair}
Let $S$ be an optimal $\SL_2$-subgroup of the classical group $G$. Then $(G,S)$ is a
Donkin pair.
\end{theorem}

\begin{proof}
 Let $S \subset G$ be an optimal $\SL_2$-subgroup, i.e.  the image of
 an optimal homomorphism $\SL_2 \to G$.  According to
 \ref{optimal-for-classical}, the natural representation $V$ of $G$
 is a restricted semisimple $S$-module. It now follows from
 \ref{exterior-of-restricted-semisimple-for-SL2} that
 $\exterior^\bullet V$ has a good filtration as $S$-module.  Finally,
 Theorem \ref{check-Donkin-pair-on-exterior-powers} now shows that
 $(G,S)$ is a Donkin pair, as required.
\end{proof}

We now wish to investigate optimal $\SL_2$-subgroups of the simply
connected covering group $\Gsc$ of the classical group $G$.  Recall by \ref{SLV-description} and \ref{SpV-description} that the classical group $G$ is simply connected unless $G =
\Omega(V,\beta)$ with $\beta$ a non-degenerate symmetric bilinear form.

\begin{stmt}
 \label{distinguished-sl2-in-symmetric-group}
 Let $\beta$ be a symmetric non-degenerate bilinear form on the
 finite dimensional vector space $V$, and let $S$ be an optimal
 $\SL_2$-subgroup of $\Omega = \Omega(V,\beta)$. Then $S$ is
 distinguished if and only if $\res^\Omega_S V$ is isomorphic as an
 $S$-module to a direct sum $\bigoplus_i V_i$ of simple $S$-submodules $V_i \simeq \ho{S}{n_i}$, where $n_i < p$ are even and
 pairwise distinct. If $S$ is distinguished, the restriction of
 $\beta$ to each $V_i$ is non-degenerate and $S \subset \prod_i
 \Omega(V_i) \subset \Omega$.
\end{stmt}

\begin{proof}
 First suppose that $\res^\Omega_S V \simeq \bigoplus_i \ho{S}{n_i}$
 with the $n_i$ as indicated. By \ref{SL2-simple-hos} the $V_i= \ho{S}{n_i}$ are simple
 self-dual $S$-submodules of $V$ which are pairwise non-isomorphic, so the
 restriction of $\beta$ to $V_i$ is non-degenerate for each $i$.
 Thus $S$ is contained in the subgroup $\prod_i
 \Omega(V_i)$.  Now, Schur's Lemma shows the endomorphism algebra
 $\End_S(V)$ to be isomorphic to $\prod_i k$ with the $i$-th factor
 acting by scalar multiplication on $V_i$. Thus the centralizer
 $T=C_{\GL(V)}(S)$ is the group of units $\End_S(V)^\times \simeq
 \prod_i \Gm$.  In particular, the $\GL(V)$-centralizer $T$ is
 contained in $\prod_i \GL(V_i)$. But the intersection of $\prod_i
 \GL(V_i)$ with $\Omega$ is precisely $\prod_i \Omega(V_i)$ and the
 intersection of $T$ with $\prod_i \Omega(V_i)$ is finite. Thus
 $C_\Omega(S)$ contains no positive dimensional torus and
 \ref{distinguished-optimal-sl2s} implies that $S$ is distinguished.

 For the converse, let $X \in \Lie(S)$ be nilpotent; since $S$ is
 assumed distinguished, the nilpotent element $X$ is
 distinguished. According to \cite{jantzen-nil}, Lemma 4.2 the
 partition describing the action of $X$ on $V$ has distinct odd
 parts. By \ref{optimal-for-classical}, $V$ is a restricted semisimple $S$-module and the required description of the $S$-module $\res^\Omega_S V$ now follows from \ref{restricted-ss-for-sl2}.
\end{proof}

\begin{theorem}
 \label{spin-group-optimal-donkin-pair}
 Let $\beta$ be a symmetric non-degenerate bilinear form on the
 finite dimensional vector space $V$, and write $\Gsc \to
 \Omega(V,\beta)$ for the simply connected covering group of
 $\Omega(V,\beta)$. Write $d$ for the rank of $\Omega(V,\beta)$; thus
 $\dim V = 2d$ or $2d+1$.  Let $S \subset \Gsc$ be an optimal
 $\SL_2$-subgroup. Then $(\Gsc,S)$ is a Donkin pair under the
 following assumptions:
 \begin{enumerate}[(a)]
 \item If $\dim V = 2d+1$ is odd, $p > \dbinom{d+1}{2} = \dfrac{(d+1)d}{2}$.
 \item If $\dim V = 2d$ is even, $p > \dbinom{d}{2} = \dfrac{d(d-1)}{2}$.
 \end{enumerate}
\end{theorem}

\begin{proof}
 Let $T$ be a maximal torus of the centralizer $C_{\Gsc}(S)$; then $S$ is contained in the subgroup $M = C_{\Gsc}(T)$, and $M$ is a Levi factor of a parabolic subgroup of $\Gsc$, hence $(\Gsc,M)$ is a Donkin pair by \ref{Levi-Donkin-pair}. It now follows from \ref{levi-of-classical} that $M \simeq M_1 \times M_2$ where $M_1$ is the simply connected covering group of a Levi factor of a parabolic subgroup of some $\SL(U)$ and $M_2$ is the simply connected covering group of $\Omega(W)$ for a non-degenerate subspace $W \subset V$.
 
 To see that $S$ is a good filtration subgroup of $M$ one may use
 \ref{direct-factor-Donkin-subgroup}; thus, it is enough to see that
 $(M_1,S)$ and $(M_2,S)$ are Donkin pairs (where by abuse of notation
 we write $S$ for its image in $M_i$). Evidentally (the image of) $S$
 is an optimal $\SL_2$-subgroup of $M_i$ for $i=1,2$. Thus Theorem
 \ref{classical-optimal-donkin-pair} already shows that $(M_1,S)$ is
 a Donkin pair.

 It remains to argue that $(M_2,S)$ is a Donkin pair. Since $T$ is a
 maximal torus of $C_{\Gsc}(S)$, evidentally $S$ is a
 \emph{distinguished} optimal $\SL_2$-subgroup of $M_2$. In
 particular, \ref{distinguished-sl2-in-symmetric-group} shows that as an
 $S$-module, $W$ is isomorphic to a direct sum $\bigoplus_i V_i$ of simple
 $S$-submodules $V_i \simeq \ho{S}{n_i}$ where $n_i < p$ are even and
 pairwise distinct. Moreover, $S$ acts on $W$ through its image in
 $\prod_i \Omega(V_i)$, hence $S$ is contained in the product
 $\prod_i \Omega(V_i)_{\sconn} \subset M_2$. According to
 \ref{direct-factor-Donkin-subgroup} to see that $(M_2,S)$ is a
 Donkin pair, it suffices to see that the image of $S$ in
 $\Omega(V_i)_{\sconn}$ is a good filtration subgroup for each $i$.
 Thus we may and will suppose that $M_2 = \Omega(W)_{\sconn}$ is the
 simply connected cover of $\Omega(W)$ with $\dim W = 2d'+1$ odd for some $d' \leq d$, and that $S$ acts irreducibly on $W$.

 Using Theorem \ref{classical-optimal-donkin-pair} together with
 \ref{donkin-subgroups-of-spin}, we see that $S$ will be a good
 filtration subgroup of $M_2$ provided that the spin module
 $L=\ho{\Omega(W)_{\sconn}}{\varpi_{d'}}$ has a good filtration as
 $S$-module. Since $\varpi_{d'}$ is a minscule weight for the root
 system $B_{d'}$, one knows that the weights of the spin representation
 are precisely the Weyl group conjugates of $\varpi_{d'}$. Using this
 description, one sees that when viewed as a module for $S$, the
 highest weight of the spin module $L$ is 
 \begin{equation*}
   \dfrac{1}{2}\left ( \sum_{i=1}^{d'} 2i \right) = \dbinom{d'+1}{2}.
 \end{equation*}
 Thus if $p > \dbinom{d'+1}{2}$, it follows from
 \ref{restricted-ss-for-sl2} that $\res^{\Omega(W)_{\sconn}}_S(L)$ is
   a restricted semisimple $S$-module, and hence has a good
   $S$-filtration.

Now, using the fact that $d' \leq d$ and that $d' < d$ if $\textrm{dim } V$ is even (since in this case $2d' + 1 = \textrm{dim } W \leq \textrm{dim } V = 2d$), the conditions on $p$ and $d$ we have given in the statement suffice to guarantee that $(\Gsc,S)$ is a Donkin pair.
\end{proof}

\begin{remark} \label{rem:table of primes for exceptional group levis}
 Let $\beta$ be a non-degenerate symmetric bilinear form on $V$, and
 let $\Omega_{\sconn}$ be the simply connected covering group of
 $\Omega(V,\beta)$. The preceding theorem shows that any optimal
 $\SL_2$-subgroup $S \subset \Omega_{\sconn}$ is a good filtration subgroup 
 under the conditions indicated in the following tables:

 \begin{tabular*}{.45\linewidth}[t]{@{\extracolsep{\fill}} |c|c|c|}
   \hline
   root system  of $\Omega(V)$ &  $\dim V$  & condition  \\
   \hline
   $D_3$ & 6  & $p \ge 5$\\
   $D_4$ & 8  & $p \ge 7$\\
   $D_5$ & 10 & $p \ge 11$ \\
   $D_6$ & 12 & $p \ge 17$ \\
   $D_7$ & 14 & $p \ge 23$ \\
   \hline
 \end{tabular*}
 \quad
 \begin{tabular*}
{.45\linewidth}[t]{@{\extracolsep{\fill}} |c|c|c|}
   \hline
   root system of $\Omega(V)$ &  $\dim V$ & condition  \\
   \hline
   $B_2$ & 5 & $p \ge 7$\\
   $B_3$ & 7 & $p \ge 11$ \\
   \hline
 \end{tabular*}
\end{remark}

%Using \ref{reduce-donkin-pair-to-simply-connected-case} we can summarize the results of this section as follows.

%\begin{theorem} \label{th:optimal sl2s of classical groups, general case} Let $G'$ be a semisimple group with no components of exceptional type. If all components of $G'$ are of type $A$ then we make no assumptions on the prime $p$. Otherwise, assume that $p > 2$ and also assume that
%\begin{enumerate}[(a)]
%\item $p > \dbinom {r+1} 2$ if $G'$ has any components of type $B_r \,$; and
%\item $p > \dbinom r 2$ if $G'$ has any components of type $D_r$.
%\end{enumerate}
%Then $(G',S)$ is a Donkin pair for every optimal $SL_2$ subgroup $S \subseteq G'$.
%\end{theorem}

\subsection{Optimal $\SL_2$-subgroups of exceptional type groups}

\label{sub:optimal sl2s in the exceptional case}
We now turn our attention to the exceptional groups.

\begin{theorem} 
\label{th:nilpotent orbit Donkin pairs} 
Let $G$ be a quasisimple group of exceptional type and let $S \subset G$ be an
optimal $\SL_2$-subgroup. Assume that $p$ is a good prime for $G$. If $S$ is distinguished then $(G,S)$ is a Donkin pair. More generally, if $S$ is an arbitrary optimal $SL_2$ subgroup then $(G,S)$ is a Donkin pair under the following additional conditions on $p$:
\begin{itemize}
\item If $G$ is of type $E_6$ or $F_4$ then $p \geq 11$;
\item If $G$ is of type $E_7$ then $p \geq 17$;
\item If $G$ is of type $E_8$ then $p \geq 23$.
\end{itemize}
\end{theorem}

\begin{proof}
Assume that $p$ is a good prime for $G$. We first prove, using a series of case-by-case computation of branching rules using the computer program LiE \cite{lie}, that $(G,S)$ is a Donkin pair for every \emph{distinguished} optimal $SL_2$ subgroup $S$ of a simply-connected quasisimple group $G$ of exceptional type (see below for an example). The result for non-simply-connected $G$ then follows from \ref{reduce-donkin-pair-to-simply-connected-case}.

Now assume that $S$ is a non-distinguished optimal $SL_2$ subgroup of $G$ and assume the additional conditions given for $p$. We start by following the technique of the proof of Theorem \ref{spin-group-optimal-donkin-pair} to reduce to the distinguished case. Let $T$ be a maximal torus of $C_G(S)$. Then $M = C_G(T)$ is a Levi factor of a parabolic subgroup of $G$ and $S$ is a distinguished optimal $SL_2$ subgroup of $M$. According to \ref{Levi-Donkin-pair}, $(G,M)$ is a Donkin pair; thus to prove the Theorem it suffices to prove that $(M,S)$ is a Donkin pair.

Now, since $S$ is assumed to not be distinguished in $G$, the derived group of $M$ has semisimple rank less than that of $G$. Hence we can proceed by induction on the rank of $M$. The case where rank $M = 1$ is trivial. By \ref{direct-factor-Donkin-subgroup} we may assume that $M$ is quasisimple. As we already know that $(M,S)$ is a Donkin pair when $M$ is of exceptional type (since $S$ is distinguished in $M$) we may assume that $M$ is not of exceptional type. Since $M$ is the Levi factor of a parabolic subgroup of an exceptional group, \ref{Levi factors of exceptional-type groups} tells us the possibilities for the type of $M$. The induction step, and hence the result, now follows from Theorem \ref{classical-optimal-donkin-pair}, Theorem \ref{spin-group-optimal-donkin-pair}, and Remark \ref{rem:table of primes for exceptional group levis}.

We conclude the proof with an illustrative example of the use of LiE for the case where $G$ is of type $E_7$ and $S$ is distinguished. For any $S$-module
 $M$, let $ \chi_S(M)$ denote the character of $M$ as an $S$-module;
 we also write $\chi_S(M) = \chi_S(\res^G_S M)$ for a $G$-module
 $M$. For any integer $n$ set $\chi_S(n):= \chi_S( \ho S { n \varpi
 }) $, where $\varpi$ denotes the single fundamental weight of $S$.

Assume that $G$ is simply-connected of type $E_7$. We use the tables starting on p. 175 of \cite{Carter} (and the labels therein) to check each distinguished $SL_2$ of $G$. For each distinguished type $S$ we verify that the required modules for $G$ given in Theorem  \ref{th:exceptional group Donkin pairs} afford restricted semisimple modules for $S$ and are hence good filtration modules for $S$. To do this, we check that the $S$-characters of these modules are sums of characters of simple induced modules $\ho S n$ for $n < p$ (cf \ref{SL2-simple-hos} and \ref{restricted-ss-for-sl2}).

  We also need to determine the primes $p$ for which there is an
  optimal $SL_2$ morphism $\phi:\SL_2 \to G$ corresponding to the
  given distinguished class. Let $X \in \Lie(B)$ be a nilpotent
  element in the distinguished orbit; then, by Theorem
  \ref{th:mcninch-optimal}, it suffices to verify that $ X^{[p]} = 0
  $. Let $ \mathfrak g = \oplus \, \mathfrak g(i) $ be the
  decomposition of $ \mathfrak g := \Lie(G) $ coming from the
  cocharacter associated to $X$; then by Proposition 24 in
  \cite{mcninch-optimal}, $ X^{[p]} = 0 $ if and only if $ \mathfrak
  g(i) = 0 $ for all $i \geq 2p$. We use this criterion to determine
  for which primes $X^{[p]} = 0$.
  \begin{itemize}
  \item When $\Lie(S)$ contains a regular nilpotent element of $G$,
    we have $\chis{ \ho G {\varpi_1} } = 35 \chis 0 + 31 \chis 2 +
    \chis 4 $ and $\chis{ \ho G {\varpi_7} } = \chis 0 + 12 \chis 2
    $. Furthermore, there is an optimal $SL_2$ homomorphism
    corresponding to the regular nilpotent orbit when $p \geq 17$.
  \item When $S$ is of type $E_7(a_1)$ or $E_7(a_5)$ we have $\chis{
      \ho G {\varpi_1} } = 17 \chis 0 + 22\chis 2 + 10 \chis 4 $ and
    $ \chis{ \ho G {\varpi_7} } = 4 \chis 0 + 14 \chis 2 + 2 \chis 4
    $. For type $E_7(a_1)$ there is an optimal $SL_2$ homomorphism
    corresponding to this distinguished type when $p \geq 13$ and for
    type $ E_7(a_5) $ we need $p \geq 5$.
  \item When $S$ is of type $E_7(a_2)$, $E_7(a_3)$, or $E_7(a_4)$ we
    have $\chis{ \ho G {\varpi_1} } = 9 \chis 0 + 15 \chis 2 +13
    \chis 4 + 2 \chis 6 $ and $\chis{ \ho G {\varpi_7} } = 8 \chis 0
    + 6 \chis 2 + 6 \chis 4 $. For types $ E_7(a_2) $ and $ E_7(a_3)
    $ there is an optimal $SL_2$ homomorphism corresponding to this
    distinguished type when $p \geq 11$; for type $ E_7(a_4) $ we
    need $p \geq 7$.
\end{itemize}
\end{proof}

%\subsection{The proof of Theorem \ref{th:nilpotent orbit Donkin pairs}}

%\subsubsection{Maximal-rank subgroups}

%\begin{itemize}
%\item Let $H$ be the maximal-rank subgroup of $G$ of type $A_1 \times
%  A_1$. Then we have $$ \chi(\varpi_1) = \chih{1,1} + \chih{2,0} $$
%  and $$ \chi( \varpi_2 ) = \chih{0,2} + \chih{2,0} + \chih{3,1} \,
%  .$$ All $H$-modules occurring here are simple for $p \geq 7$.
%\item Let $H$ be the maximal-rank subgroup of $G$ of type $A_2$. Then
%  we have $$ \chi( \varpi_1 ) = \chih{0,0} + \chih{0,1} +
%  \chih{1,0} $$ and $$ \chi( \varpi_2 ) = \chih{0,1} + \chih{1,0} +
%  \chih{1,1} \, .$$ All $H$-modules occurring here are simple for $p
%  \geq 5$.
%\end{itemize}

\subsection{Subsystem subgroups in the exceptional cases} 
\label{sub:subsystem subgroups in the exceptional case}

Let $G$ be a semisimple group. Recall that a \emph{subsystem subgroup} of a semisimple group $G$ is a connected semisimple subgroup which is normalized by a maximal torus of $G$. In this section we consider pairs $(G,H)$ where $H$ is a subsystem subgroup of $G$. Remark that Levi factors of parabolic subgroups of $G$ are subsystem subgroups; these Levi factors are good filtration subgroups by \ref{Levi-Donkin-pair}. 

\begin{remark} [cf \cite{lieb}] \label{rem:subsystem subgroups} Subsystem subgroups of $G$ can be characterized as follows. Recall that $R$ is the root system of $G$. For any sub-root system $R' \subseteq R$ let $H$ be the subgroup of $G$ generated by $T$ and the root subgroups $\{ U_\beta : \beta \in R' \}$. Then $H$ is a subsystem subgroup of $G$ and every subsystem subgroup of $G$ can be obtained in this way.

There is also a nice inductive procedure for constructing all subsystem subgroups of $G$. Starting with the extended Dynkin diagram of $G$, remove any collection of nodes. This gives the Dynkin diagram $D'$ of a subsystem subgroup, and there is one conjugacy class of subsystem subgroups for each subdiagram obtained in this way. Now repeat this process with the connected components of the subdiagram $D'$. In this fashion one obtains all subsystem subgroups of $G$. In particular, there are only finitely many conjugacy classes of subsystem subgroups.
\end{remark}

A subsystem subgroup of maximal rank (= rank $G$) is the centralizer of an involution if and only if it is obtained via the technique of Remark \ref{rem:subsystem subgroups} by removing a node from the extended Dynkin diagram with label 2 (where we label the extended Dynkin diagram based on the coefficients occurring in the highest root). We first recall the following theorem. 

\begin{theorem}[\cite{brundan}, Proposition 3.3 and \cite{vdK}, \S6]  \label{th:brundan vdk Donkin pairs} Let $H$ be a reductive subgroup of $G$ which is the centralizer of an involution. If not all components of $G$ are of type $A_r$ or $C_r$ assume that $p > 2$. Then $(G,H)$ is a Donkin pair. In particular, if $G$ has no components of exceptional type then $(G,H)$ is a Donkin pair for every subsystem subgroup $H$ of $G$.
\end{theorem}

Recall that if we have reductive subgroups $H \subseteq H' \subseteq G$ such that $(G,H')$ and $(H', H)$ are Donkin pairs, then also $(G,H)$ is a Donkin pair. Thus, to determine for which primes a given subsystem subgroup $H \subset G$ will be a good filtration subgroup, it suffices to consider the case where $H$ has maximal rank (= rank $G$). In the following theorem we consider the remaining subsystem subgroups of maximal rank which are not covered by Theorem \ref{th:brundan vdk Donkin pairs}. This list can be obtained by the method of Remark \ref{rem:subsystem subgroups} above.

\begin{theorem} \label{th:subsystem subgroups in the exceptional case}  Suppose that $p>2$, let G be a quasisimple group of exceptional type and let $H \subset G$ be a subsystem subgroup of maximal rank such that $H$ is not the centralizer of an involution. Then the pair $(G,H)$ appears in the following list, and $H$ is a good filtration subgroup of $G$ provided that $p$ satisfies the indicated condition.

\begin{itemize}
\item $G$ is simply-connected of type $F_4$
\begin{itemize}
	\item{$A_2 \times A_2$}: $p \geq 5$.
	\item{$A_3 \times A_1$}: $p \geq 5$.
\end{itemize}

\item $G$ is simply-connected of type $G_2$
\begin{itemize}
	\item{$A_2 $}: $p \neq 3$.
\end{itemize}

\item $G$ is simply-connected of type $E_6$
\begin{itemize}
	\item{$A_2 \times A_2 \times A_2 $}: $p \geq 5$.
\end{itemize}

\item $G$ is simply-connected of type $E_7$
	\begin{itemize}
	\item{$A_5 \times A_2$}: $p \notin \{  2, 5, 7  \}$.
%	\item{The other $A_5 \times A_2 $}: Is it conjugate to the first one?
	\end{itemize}
	
\item $G$ is simply-connected of type $E_8$
	\begin{itemize}
	\item{$A_8 $}: $p \geq 7$.
	\item{$A_1 \times A_2 \times A_5$}: $p \geq 7$.
	\item{$A_4 \times A_4 $}: $p \geq 7$.
	\item{$D_5 \times A_3$}: $p \geq 7$. 
	\item{$E_6 \times A_2$}: $p \geq 7$.
	\item{$A_1 \times A_7$}: $p \geq 11$.
	\end{itemize}
\end{itemize}

\end{theorem}

\begin{proof}
As with Theorem \ref{th:nilpotent orbit Donkin pairs}, this theorem is proved by a case-by-case computation using branching rules in LiE. For the given group $G$ of exceptional type and subsystem subgroup $H$ we first use LiE to compute the characters of the appropriate modules from Theorem \ref{th:exceptional group Donkin pairs} considered as $H$-modules. This character is written as a sum of characters of induced modules for $H$. We then find the minimal prime $p'$ such that the highest weights of these induced modules for $H$ lie in the low alcove. This implies that for $p \geq p'$ the $G$-modules from Theorem \ref{th:exceptional group Donkin pairs} afford simple induced modules for $H$ and thus are good filtration modules for $H$. Hence that theorem implies that $(G,H)$ is a Donkin pair for $p \geq p'$. In addition, we can sometimes extend our analysis using linkage, as indicated in the example below.

As in the proof of Theorem \ref{th:nilpotent orbit Donkin pairs}, we will not give the details of each computation; we will instead give one illustrative example.

Let $G$ be simply-connected of type $E_8$ and let $H \subset G$ be
the subsystem subgroup of type $D_5 \times A_3$. For a dominant weight $\mu$ of $H$ let $\chih \mu$ denote the character of the induced module $\ho H \mu$ with highest weight $\mu$. We write the weights of
$D_5 \times A_3$ as $(a, b, c, d, e, f, g, h)$, where $(a, b, c, d,
e)$ is considered as a weight of $D_5$ and $(f, g, h)$ is considered
as a weight of $A_3$. Then the character of $\ho G{\varpi_1} $as
an $H$-module is $\displaystyle \sum_{\mu \in C} \chih \mu$, where $C$ is the following set of dominant weights for $H$:
\begin{align*}
&\big\{ (0) , (0,0,0,0,0,0,2,0) ,
(0,0,0,0,0,1,0,1) , (0,0,0,0,1,0,1,1) , \\
&(0,0,0,0,1,1,0,0) , (0,0,0,1,0,0,0,1) ,
(0,0,0,1,0,1,1,0) , (0,0,0,1,1,0,0,0) , \\
&(0,0,1,0,0,0,1,0) , (0,1,0,0,0,1,0,1)
(1,0,0,0,0,0,0,2) , (1,0,0,0,0,0,1,0) , \\
&(1,0,0,0,0,2,0,0) , (1,0,0,0,1,0,0,1) ,
(1,0,0,1,0,1,0,0) , (2,0,0,0,0,0,0,0) \big\}.
\end{align*}
Also, the character
of $\ho G{\varpi_8} $ as an $H$-module is $\displaystyle \sum_{\mu \in C'} \chih \mu$, where $C'$ is the following set of dominant weights for $H$:
\begin{equation*}
\big\{ (0,0,0,0,0,1,0,1) ,
(0,0,0,0,1,1,0,0) , (0,0,0,1,0,0,0,1) ,
(0,1,0,0,0,0,0,0) , (1,0,0,0,0,0,1,0) \big\}. 
\end{equation*}

We now consider the weights in $C \cup C'$ to find primes for which
the associated induced modules for $H$ are all simple. For example, consider the
weight $( 0, 0, 0, 1, 0, 1, 1, 0 ) \in C$. The induced module with highest
weight $(0, 0, 0, 1, 0)$ for $D_5$ is in the low alcove for $D_5$
when $p \geq 11$ and the induced module with highest weight $(1, 1,
0)$ for $A_3$ is in the low alcove for $A_3$ when $p \geq
5$. Thus $\ho H { 0,0,0,1,0,1,1,0 }$ is simple when $p \geq 11$.

Checking all the above weights in this manner, we see that for all $\mu \in C \cup C'$, $\chih \mu$ is the character of a simple $H$-module
for $p \geq 11$.

To further extend the analysis we consider linkage. We check to see if we can find a prime $< 11$ so that $\mu$ is the minimal dominant weight in its linkage class for all $\mu \in C \cup C'$; if so, $\chih \mu$ will be the character of a simple $H$-module for that prime also. This linkage computation now shows that for $p=7$, $\chih \mu$ is the character of a simple induced module for all $\mu \in C \cup C'$. Thus $H$ is a good filtration subgroup of $G$ for $p \geq 7$.
\end{proof}

\begin{remark} In Theorem \ref{th:subsystem subgroups in the exceptional
   case}, the primes listed give sufficient but not, a priori,
 necessary conditions for the given subsystem subgroup to be a good
 filtration subgroup. One could perhaps extend the results in the
 theorem to more primes using a finer analysis of alcove
 considerations and linkage or by using
 \ref{Donkin-pair-by-set-of-weights} rather than Theorem
 \ref{th:exceptional group Donkin pairs}.
\end{remark}

\bibliographystyle{model1-num-names}
\bibliography{bibliography}

%\begin{thebibliography}{00}
%\end{thebibliography}

\end{document}